\documentclass[12pt]{amsart}
\usepackage{amsmath,amssymb,amsthm,amscd}
\numberwithin{equation}{section}

\def\tg{\widetilde g}

\def\R{{\mathfrak R}}

\def\th{\tilde h}

\newtheorem{theo}{Theorem}[section]
\newtheorem{lemm}{Lemma}[section]
\newtheorem{coro}{Corollary}[section]
\newtheorem{rema}{Remark}[section]

\newtheorem{defi}{Definition}[section]

\def\begeq{\begin{equation}}
\def\endeq{\end{equation}}
\def\p{\partial}

\def\lf{\left}
\def\ri{\right}

\def\R{\Bbb R}
\def\wt{\widetilde}

\def\hg{\hat g}
\def\hh{\hat h}
\def\hn{\hat n}
\def\hnabla{\hat\nabla}

\begin{document}
\title{large-sphere and small-sphere  limits of the Brown-York mass}
\author{Xu-Qian Fan}
\address{Department of Mathematics, Jinan University, Guangzhou,
 510632,
P. R. China.}
\author{Yuguang Shi$^1$}
\thanks{$^1$Research partially supported by 973 Program (2006CB805905)
 of China   and
Fok YingTong Education Foundation.}

\address{Key Laboratory of Pure and Applied mathematics, School of
 Mathematics Science, Peking University,
Beijing, 100871, P.R. China.} \email{ygshi@math.pku.edu.cn}

\author{Luen-Fai Tam$^2$}
\thanks{$^2$Research partially supported by Earmarked Grant of Hong
Kong \#CUHK403005.}
\address{The Institute of Mathematical Sciences and Department of
 Mathematics, The Chinese University of Hong Kong,
Shatin, Hong Kong, China.} \email{lftam@math.cuhk.edu.hk}
\renewcommand{\subjclassname}{%
  \textup{2000} Mathematics Subject Classification}
\date{November 2007}
\begin{abstract}
In this paper,  we will   study the limiting behavior of the
Brown-York mass of the coordinate spheres  in an asymptotically
flat manifold. Limiting behaviors  of volumes of regions related
to coordinate spheres are also obtained, including a discussion on
the isoperimetric mass  introduced by Huisken \cite{Huisken}.  We
will also study
 expansions of the Brown-York
mass and the Hawking mass of   geodesic spheres with center at a
fixed point $p$ of a three manifold. Some geometric consequences
will be derived.
\end{abstract}
\maketitle\markboth{Xu-Qian Fan, Yuguang Shi and Luen-Fai
Tam}{large-sphere and small-sphere  limits}

\section{Introduction}
In this work, we will discuss the large-sphere limit  of
  the Brown-York mass in an asymptotically flat manifold and the
  small-sphere limit  of
  the Brown-York mass near a point in a three dimensional
  manifold. We will also discuss the behaviors of large-sphere limit and
  small-sphere limit of other interesting quantities.

   Let us first recall some definitions. In general
relativity, asymptotically flat manifolds have great interests in
many problems. In this paper, we adopt the following definition of
asymptotically flat manifolds.

\begin{defi}\label{defaf}
A complete three manifold $(M,g)$ is said to be asymptotically flat
(AF) of order $\tau$ (with one end) if there is a compact subset $K$
such that $M\setminus K$ is diffeomorphic to $\R^3\setminus B_R(0)$
for some $R>0$ and in the standard coordinates in $\R^3$, the metric
$g$ satisfies:
\begin{equation} \label{daf1}
g_{ij}=\delta_{ij}+\sigma_{ij}
\end{equation}
with
\begin{equation} \label{daf2}
|\sigma_{ij}|+r|\p \sigma_{ij}|+r^2|\p\p\sigma_{ij}|+r^3|\p\p\p
\sigma_{ij}|=O(r^{-\tau}),
\end{equation}
for some constant $1\ge\tau>\frac{1}{2}$, where $r$ and $\p$ denote
the Euclidean distance and standard derivative operator on $\R^3$
respectively.
\end{defi}
A coordinate system of $M$ near infinity so that the metric tensor
in these coordinates satisfies the decay conditions in the
definition is said to be {\it admissible}. Note that some of the
results in the following do not need decays of the third order
derivatives of $\sigma_{ij}$.

\begin{defi}
The Arnowitt-Deser-Misner (ADM) mass (see \cite{ADM}) of an
asymptotically flat manifold $M$ is defined as:
\begin{equation} \label{defadm1}
m_{ADM}(M)=\lim_{r\to\infty}\frac{1}{16\pi}\int_{S_r}
\lf(g_{ij,i}-g_{ii,j}\ri)\nu^jd\Sigma_r^0,
\end{equation}
where $S_r$ is the Euclidean sphere, $d\Sigma_r^0$ is the volume
element induced by the Euclidean metric, $\nu$ is the outward unit
normal of $S_r$ in $\R^3$ and the derivative is the ordinary partial
derivative.
\end{defi}
We always assume that the scalar curvature is in $L^1(M)$ so that
the limit exists in the definition. Under the decay conditions in
the definition of AF manifold, the definition of ADM mass  is
independent of the choice of admissible coordinates  by the result
of Bartnik \cite{BTK86}.

 Let $\lf(\Omega, g\ri)$ be a compact three manifold with
smooth boundary $\p \Omega$. Suppose the Gauss curvature of
$\p\Omega$ is positive, then the Brown-York quasi-local mass of
$\p\Omega$ is defined as  (see \cite{BY1,BY2}):
\begin{defi}
\begin{equation} \label{defbym1}
m_{BY}\lf(\p \Omega\ri)=\frac{1}{8\pi}\int_{\p \Omega}(H_0-H)d\Sigma
\end{equation}
 where  $H$ is the mean curvature of $\p\Omega$ with respect to the
 outward unit normal and
  the metric $g$, $d\Sigma$ is the volume element induced on $\p\Omega$ by
 $g$ and $H_0$ is the mean curvature  of $\p \Omega$ when
  embedded in  $\R^3$.
\end{defi}
The Brown-York mass is well-defined because by the result of
Nirenberg \cite{Niren}, $\p\Omega$ can be isometrically embedded in
$\R^3$ and the embedding is unique by
\cite{Herglotz43,Sacksteder62,PAV}. In particular, $H_0$ is
completely determined by the metric on $\p\Omega$. However, this is
a global property.  In contrast, the norm of the mean curvature
vector of an embedding of $\p\Omega$ into the light cone in the
Minkowski space can be expressed explicitly in terms of the Gauss
curvature, see \cite{BLY99}. Hence in the study of Brown-York mass,
one of the difficulties  is to estimate $\int_{\p\Omega}H_0d\Sigma$.
We will use the Minkowski formulae \cite{KLBG} and the estimates of
Nirenberg \cite{Niren} in this regard.

In the first part of this  paper, we want to study limiting
behaviors of Brown-York mass on large spheres.  We will verify the
following:

\begin{theo}\label{thmlc-in}
Let $(M, g )$ be an asymptotically flat manifold of order
$\tau>\frac12$ with one end and let $S_r$ be the coordinate
spheres in some admissible coordinates. Then
$$
\lim_{r\rightarrow \infty}m_{BY}(S_r)=m_{ADM}(M).
$$
Here $m_{BY}(S_r)$ is the Brown-York quasi-local mass of $S_r$, and
$m_{ADM}(M)$ is the Arnowitt-Deser-Misner (ADM) mass of $M$.
\end{theo}

Theorem \ref{thmlc-in} was observed and proved to be true by many
people, see the works of Brown-York \cite{BY2}, Hawking-Horowitz
\cite{HKGH1}, Braden-Brown-Whiting-York \cite{BBWYY} and
Baskaran-Lau-Petrov \cite{BLP}, see also \cite{ST02}. However, in
this paper, we will use a different method to derive Theorem
\ref{thmlc-in}. Interestingly, our method leads to the following
volume comparison result. Let $V(r)$ be the volume with respect to
an AF metric $g$ of the region inside $S_r$ and let $V_0(r)$ be the
Euclidean volume inside the surface $S_r$ when embedded in $\R^3$.
\begin{theo}\label{vol-comparison-L-in}
 Let $(M,g)$ be an asymptotically flat
manifold of order $\tau>\frac12$ with one end.   Then
\begin{equation}\label{volcompL-e1-in}
    V_0(r)-V(r)=-2m_{ADM}(M)\pi r^2+o(r^2).
\end{equation}
\end{theo}
Hence if the ADM mass is nonnegative, then
$\lim_{r\to\infty}r^{-2}(V(r)-V_0(r))\ge0$. Combining this with
Positive Mass Theorem, if we further assume that the scalar
curvature is nonnegative, then the limit is zero if and only if
$M$ is isometric to $\R^3$.

In \cite{Huisken}, a notion of {\it isoperimetric mass} $m_{ISO}(M)$
of an AF manifold is introduced by Huisken. It is defined as:

$$
m_{ISO}=\limsup_{r\to\infty}\frac2{\mathcal{A}(r)}\lf(V(r)-\frac1{6\pi^\frac12}\mathcal{A}^\frac32(r)\ri)
$$
where $V(r)$ is as before and $\mathcal{A}(r)$ is the area of the
coordinate sphere with respect to the AF metric. Using the method of
proof of Theorem \ref{vol-comparison-L-in}, Miao \cite{Miao} proves
that the isoperimetric mass and the ADM mass of an AF manifold are
equal. We will include Miao's result in this work.

 In the second part of the paper, we will consider the small
sphere limit of the Brown-York mass. Let $r$ be the distance to
the fixed point $p$, and $R(p)$ is the scalar curvature evaluated
at $p$. We have the following:

\begin{theo}\label{smallspherelimit-in}
Let $(N,g)$ be a Riemannian manifold of dimension three, $p$ be a
fixed interior point on $N$, and $S_r$ be the geodesic sphere of
radius $r$ center at $p$. For $r$ small enough, we have
\begin{equation}\label{smallspherelimit-e1-in}
m_{BY}(S_r)=\frac{r^3}{12}R(p)+\frac{r^5}{1440}\lf[24|Ric|^2
(p)-13R^{2}(p) + 12 \Delta R (p)\ri]+O(r^6),
\end{equation}
where, $\Delta$ is Laplacian operator of $(M,g)$ and $|Ric|$ is
the norm of the Ricci curvature.
\end{theo}
Let  $M$ be an AF manifold with nonnegative scalar curvature.
Suppose the Brown-York mass of the coordinate spheres converge to
zero, then $M$ must be the Euclidean space by Theorem
\ref{thmlc-in}
  and the Positive Mass Theorem in \cite{SYL79, Wit81}.
  By Theorem \ref{smallspherelimit-in}, we have similar result near a
  point $p$. Namely, assume   $R\ge 0$ in a neighborhood of $p$,
then
 \begin{equation}\label{nonnegative}
 \lim_{r\rightarrow 0}\frac{m_{BY}(S_r)}{r^5}\ge 0.
 \end{equation}
Equality holds if and only if  $(N,g)$ is flat at $p$ and $R$
vanishes up to second order at $p$.

There are  results on the small-sphere limits obtained by
Brown-Lau-York \cite{BLY99}. They consider a cut $S_r$ with an
affine radius $r$ of the light cone at a point $p$ in a Lorentz
manifold. Using the light cone of reference, they show that the
expansion of the  quasi-local energy is:
$$
E=\frac{4\pi r^3}{3}T_{ab}n^an^b+o(r^3),
$$
where $T_{ab}$ is the energy momentum tensor and $n$ is the unit
future pointing time like vector defining the choice of the affine
parameter. In our case, if we consider the Lorentz manifold
$\R\times N$ with metric $\wt g=-dt^2+g$, and suppose the metric
satisfies the Einstein equation:
$$
\wt R_{ab}-\frac12 \wt R\wt g_{ab}=8\pi T_{ab}.
$$
Let $n=\frac{\p}{\p t}$ be the future pointing unit normal, then

$$
\frac{R(p)}{12}=\frac{4\pi r^3}{3}T_{ab}n^an^b.
$$
Hence $r^3$ term of the expansion in our case is similar to that
in \cite{BLY99}. However, we are using Euclidean reference and we
only consider the time symmetric case.

In the case of   vacuum space-time,  Brown-Lau-York \cite{BLY99}
also obtain the $r^5$ term in the expansion of $E$ as follows:

$$
E_5=\frac{r^5}{90}T_{abcd}n^an^ba^cn^d
$$
where $T_{abcd}$ is the Bel-Robinson tensor, which depends only on
the curvature tensor (and the metric). In Theorem
\ref{smallspherelimit}, the space-time is not vacuum in general
and is time symmetric. The coefficient  of the term $r^5$ depends
not only on the curvature tensor, but also on the derivative of
the scalar curvature. For the sake of comparison, in our case, one
can compute that $T_{0000}=\frac18(4|Ric|^2-R^2)$. We use the
definition of Bel-Robinson tensor as in (5) of \cite{Douglas}

Next we want to compare   the expansion of the Hawking mass with the
expansion of the Brown-York mass for small spheres. Recall the
definition of the Hawking mass. Let $(\Omega,g)$ be a smooth three
manifold with boundary $\p\Omega$ and let $H$ be the mean curvature
on $\p \Omega$ with respect to the outward unit normal, the Hawking
quasi-local mass is  defined as (see \cite{HKG}):
\begin{defi}
\begin{equation} \label{defhkgm1}
m_H(\p \Omega)=\frac{|\p
\Omega|^{1/2}}{(16\pi)^{3/2}}\lf(16\pi-\int_{\p
\Omega}H^2d\Sigma\ri)
\end{equation}
 where $d\Sigma$ is the volume element induced on $\p\Omega$ by
 $g$
 and $|\p \Omega|$ is the area of $\p \Omega$.
\end{defi}
With  the same notations and assumptions in Theorem
\ref{smallspherelimit-in}, the expansion of $m_H(S_r)$ is given
by:
\begin{equation}
\begin{split}
m_H (S_r)=\frac{r^3}{12}R(p)+\frac{r^5}{720}\lf(6\Delta R (p)-5
R^2(p)\ri)+O(r^6).
\end{split}
\end{equation}
One can see that $m_{BY}(S_r)$ and $m_H(S_r)$ are equal up to the
term with order $r^3$. However, the terms of order $r^5$ are
different. In particular, if the scalar curvature is zero near
$p$, but it is non-flat at $p$, then $r^{-5}m_{BY}(S_r)>0$ and
$m_H(S_r)=O(r^6)$ for small $r$.

As in the large-sphere case, one can also   compare $V(r)$ and
$V_0(r)$, where $V(r)$ is the volume of the geodesic ball of
radius $r$ at $p$ and $V_0(r)$ which is the volume of the region
bounded by $S_r$ when embedded in $\R^3$.

The paper is organized as follows. In Section 2, the limit of
behavior of Brown-York mass in large sphere and volume comparison
are proved; in Section 3, small sphere limit of Brown-York mass and
Hawking mass and small sphere  volume comparison  are proved.

The authors would like to thank Robert Bartnik, Yanyan Li and
Pengzi Miao for  useful discussions.

\section{large-sphere  limit}
 In this section, we will first prove the following theorem (Theorem
 \ref{thmlc-in}).

 \begin{theo} \label{thmlc}
Let $(M,g)$ be an asymptotically flat manifold of order
$\tau>\frac12$ with one end and let $S_r$ be the coordinate
spheres in some admissible coordinates. Then
$$
\lim_{r\to\infty}m_{BY}(S_r)=m_{ADM}(M).
$$
Here $m_{BY}(S_r)$ is the Brown-York quasi-local mass of $S_r$ and
$m_{ADM}(M)$ is the Arnowitt-Deser-Misner (ADM) mass of $M$.
\end{theo}

  Consider an AF manifold $(M,g)$ with
coordinates $(x^1,x^2,x^3)$ so that $g_{ij}$ satisfies the decay
conditions in Definition \ref{defadm1}. Let $n=n^i\frac{\p}{\p
x^i}$ be the unit outward normal of $S_r$ and $n_i=g_{ij}n^j$.
Then
\begin{equation}\label{outwardnormal1}
   n^i=\frac{g^{ij}x^j}{r|\nabla r|}\ \text{and}\
   n_i=\frac{x^i}{r|\nabla r|}
\end{equation}
where $r=\lf(\sum_{i=1}^3\lf(x^i\ri)^2\ri)^\frac12$. The metric
induced on $S_r$ is $h_{ij}=g_{ij}-n_in_j$ and the second
fundamental form is $A_{ij}=h^k_ih^l_j  n_{k;l},$ where $ n_{k;l}$
is the covariant derivative of $n_k$ with respect to $g$.

\begin{lemm}\label{2nd}
With the above notations and assumptions, on $S_r$ we have the
following:
\begin{itemize}
  \item [(i)]
  $$
  A_{ij}=\frac{h_{ij}}r+O(r^{-1-\tau}), H=\frac 2
  r+O(r^{-1-\tau}), K=\frac{1}{r^2}+O(r^{-2-\tau}),
  $$
   where $H$ is the mean curvature and  $K$ is the  Gauss
   curvature   of $S_r$.
  \item[(ii)]
  $$
  d\Sigma_r=\lf(1+h^{ij}\sigma_{ij}+O(r^{-2\tau})\ri)^\frac12d\Sigma_r^0.
  $$
  Hence
  $$
  \mathcal{A}(r)=4\pi r^2+\frac12\int_{S_r}h^{ij}\sigma_{ij}
  d\Sigma_r+O(r^{2-2\tau}),
  $$
where $\mathcal{A}(r)$ is the area of $S_r$ with respect to $g$.
\end{itemize}
\end{lemm}
\begin{proof} (i)  is well-known, see
\cite{HI}. For the sake of completeness, we derive it as follows:
\begin{equation}\label{2nd-eq0}
   |\nabla r|^2=1-\frac{\sigma_{ij}x^ix^j}{r^2}+O(r^{-2\tau})
   \end{equation}
   and
\begin{equation}\label{2nd-eq1}
   \begin{split}
     \frac{\p}{\p x^k}\lf(|\nabla r|^2\ri) & =
     \frac{\p}{\p x^k}\lf(g^{ij}\frac{x^ix^j}{r^2}\ri) \\
       & =\frac{\p}{\p
       x^k}\lf[1+\lf(g^{ij}-\delta_{ij}\ri)\frac{x^ix^j}{r^2}\ri]\\
       &=O(r^{-1-\tau}).
   \end{split}
\end{equation}
So
\begin{equation}\label{outwardnormal2}
   n^i=\frac{x^i}r+O(r^{-\tau})
\end{equation}
and
\begin{equation}\label{cdofniwetl}
\begin{split}
  n_{i;j} & =\frac{\p n_i}{\p x^j}-\Gamma_{ij}^k n_k\\
         &=\frac{\p }{\p x^j}\lf(\frac{x^i}{r|\nabla
         r|}\ri)+O(r^{-1-\tau})\\
         &=\lf(\frac{\delta_{ij}}r-\frac{x^ix^j}{r^3}\ri)+O(r^{-1-\tau}).
   \end{split}
\end{equation}
where  $\Gamma_{ij}^k$ are the Christoffel symbols. Let
$h_i^j=g^{jk}h_{ki}$. Using the fact that $n$ has unit length, we
have
\begin{equation}\label{2nd-eq2}
    \begin{split}
       A_{ij}-\frac{h_{ij}}r
       &=h_i^kh^l_j n_{k;l}-\frac{h_{ij}}r\\
       &=h^l_j  n_{i;l}-\frac{h_{ij}}r\\
       & = n_{i;j}-\lf(\frac{\delta_{ij}}r-\frac{x^ix^j}{r^3}\ri)+O(r^{-1-\tau})\\
         &=O(r^{-1-\tau}).
     \end{split}
\end{equation}

From this and the fact that the curvature of $M$ decays like
$r^{-2-\tau}$, the estimates of $H$ and $K$ follows.

(ii) Let $e_1$ and $e_2$ be orthonormal frames on $S_r$ with
respect to the Euclidean metric, then
\begin{equation}\label{area-eq1}
    \begin{split}
      d\Sigma_r & =
      \lf(g(e_1,e_1)g(e_2,e_2)-g^2(e_1,e_2)\ri)^\frac12d\Sigma_r^0\\
        &
  =\lf(1+\sigma(e_1,e_1)+\sigma(e_2,e_2)+O(r^{-2\tau})\ri)^\frac12d\Sigma_r^0\\
        &=\lf[1+\lf(e_1(x^i)e_1(x^j)+e_2(x^i)e_2(x^j)\ri)
        \sigma_{ij} +O(r^{-2\tau})\ri]^\frac12d\Sigma_r^0\\
        &=\lf[1+\lf(\nabla_0x^i\cdot \nabla_0x^j-\frac{\p x^i}{\p
 r}\frac{\p x^j}{\p
        r}\ri)\sigma_{ij} +O(r^{-2\tau})\ri]^\frac12d\Sigma_r^0\\
        &=\lf[1+\lf(\delta_{ij}-\frac{\p x^i}{\p r}\frac{\p x^j}{\p
        r}\ri)\sigma_{ij} +O(r^{-2\tau})\ri]^\frac12d\Sigma_r^0\\
        &=\lf(1+h^{ij}\sigma_{ij}+O(r^{-2\tau})\ri)^\frac12d\Sigma_r^0
    \end{split}
\end{equation}
where $\nabla_0$ is the derivative with respect to the Euclidean
metric and `$\cdot$' is the standard inner product in $\R^3$. The
last statement follows from this immediately.
\end{proof}

\begin{lemm}\label{meancurvint}
  \begin{equation}\label{meancurvint-e1}
   \int_{S_r}Hd\Sigma_r=\frac {
  \mathcal{A}(r)}r+4\pi r-8\pi m_{ADM}(M)+o(1)
  \end{equation}
as $r\to\infty$.
\end{lemm}

\begin{proof}
Let $m=m_{ADM}(M)$. By Lemma \ref{2nd} and the first variational
formula, we have
\begin{equation}\label{meancurvint-e2}
    \begin{split}\frac{d}{dr}\mathcal{A}(r)&=
    \int_{S_r}\frac1{|\nabla r|}Hd\Sigma_r\\
    &=\int_{S_r}Hd\Sigma_r+
    \int_{S_r}\frac{\sigma_{ij}x^ix^j}{r^3}d\Sigma_r+O(r^{1-2\tau})
    \end{split}
\end{equation}
where we have used \eqref{2nd-eq0}.

 On the other hand, by Lemma \ref{2nd}, we have

\begin{equation}\label{meancurvint-e3}
    \begin{split}\frac{d}{dr}\mathcal{A}(r)&=
    8\pi r+\frac12\int_{S_r}\frac{\p }{\p r}\lf(h^{ij}\sigma_{ij}\ri)d\Sigma_r
    +\frac1r\int_{S_r}h^{ij}\sigma_{ij}d\Sigma_r+O(r^{1-2\tau})\\
    &=8\pi
    r+\frac12\int_{S_r}h^{ij}\sigma_{ij,k}\frac{x^k}rd\Sigma_r+
    \frac1r\int_{S_r}h^{ij}\sigma_{ij}d\Sigma_r+O(r^{1-2\tau})\\
    &=8\pi r+\frac12
    \int_{S_r} \frac{\sigma_{ii,k}x^k}rd\Sigma_r^0-\frac12
    \int_{S_r} \frac{\sigma_{ij,k}x^ix^jx^k}{r^3}d\Sigma_r^0+
    \frac1r\int_{S_r}h^{ij}\sigma_{ij}d\Sigma_r+O(r^{1-2\tau}),
    \end{split}
\end{equation}
where $\sigma_{ij,k}=\frac{\p \sigma_{ij}}{\p x^k}$.  Now, as in
\cite[(5.17)]{HI}:
\begin{equation}\label{meancurvint-e4}
   \begin{split}
 \int_{S_r}
 &\frac{\sigma_{ij,k}x^ix^jx^k}{r^3}d\Sigma_r^0\\
 &=\int_{S_r}\frac{\p}{\p
 x^k}\lf(\frac{\sigma_{ij}x^j}r\ri)\frac{x^ix^k}{r^2}d\Sigma_r^0\\
 &=-\int_{S_r} \lf(\delta_{ik}-\frac{x^ix^k}{r^2}\ri)\frac{\p}{\p
 x^k}\lf(\frac{\sigma_{ij}x^j}r\ri)d\Sigma_r^0+
 \int_{S_r}  \frac{\p}{\p
 x^i}\lf(\frac{\sigma_{ij}x^j}r\ri)d\Sigma_r^0\\
 &=-2\int_{S_r}\frac{\sigma_{ij}x^ix^j}{r^3}d\Sigma_r^0+\int_{S_r}\frac{\sigma_{ij,i}x^j}rd\Sigma_r^0+\int_{S_r}\sigma_{ij}\lf(\frac{\delta_{ij}}r-\frac{x^ix^j}{r^3}\ri)d\Sigma_r^0\\
 &=-2\int_{S_r}\frac{\sigma_{ij}x^ix^j}{r^3}d\Sigma_r+\int_{S_r}\frac{\sigma_{ij,i}x^j}rd\Sigma_r^0+\frac1r\int_{S_r}h^{ij}\sigma_{ij}d\Sigma_r+O(r^{1-2\tau}).
\end{split}
\end{equation}
Combining this with \eqref{meancurvint-e3},  by Lemma \ref{2nd}
and the definition of ADM mass, we have:
\begin{equation}\label{meancurvint-e5}
\frac{d}{dr}\mathcal{A}(r)=\frac{\mathcal{A}(r)}r+4\pi r+
\int_{S_r}\frac{\sigma_{ij}x^ix^j}{r^3}d\Sigma_r^0-8\pi m+o(1).
\end{equation}

By \eqref{meancurvint-e2} and  \eqref{meancurvint-e5}, the
lemma follows.
\end{proof}

By Lemma \ref{2nd}, if $r$ is large enough, then the Gauss curvature
of $S_r$ is positive. So  $S_r$ can be isometrically embedded in
$\R^3$ uniquely up to an isometry of $\R^3$ by
\cite{Niren,Herglotz43,Sacksteder62,PAV}. The following lemma says
that the embedded surface (rescaled) is very close to the standard
sphere as $r\to\infty$.

\begin{lemm} \label{isomr1}
  Let $(M,g)$ be an AF three manifold with \eqref{daf1} and
\eqref{daf2} for $\tau>\frac12$, and  let $S_r$ be coordinate
spheres. For $r$ large enough, there is an isometrical embedding
$X_r$ of $S_r$ in $\R^3$
 such that:
\begin{equation} \label{lisomreq1}
\begin{split}
X_r\cdot  {n_0} &=r+O\lf(r^{1-\tau}\ri) \\
H_0&= \frac{2}{r}+H_1 \text{ with } H_1=O\lf(r^{-1-\tau}\ri)
\end{split}
\end{equation}
as  $r\to +\infty$, where $ {n_0}$ is the   unit outward normal to
the surface $X_r$, `$\cdot$' is the inner product in $\R^3$, and
$H_0$ is the mean curvature of $X_r.$
\end{lemm}

\begin{proof}
For $r>0$, define a map $x=ry$ and pull back the metric to the $y$
space. Let the pull back metric be $\hat g$. Let $\hat h$ be the
induced metric on the coordinate spheres in $y$.
\begin{equation}\label{embed1}
   \begin{split}
     \hh_{ij} &=\hg_{ij}-\hn_i\hn_j\\
       & =r^2 g_{ij}-\hn_i\hn_j
   \end{split}
\end{equation}
where $\hh_{ij}=\hh(\frac{\p}{\p y^i},\frac{\p}{\p y^j})$, etc.
and $g_{ij}=g(\frac{\p}{\p x^i},\frac{\p}{\p x^j})$ etc. Also
$\hn_i=y^i/(\rho|\hnabla \rho|_{\hg})$ is the unit normal on
$\lf\{\rho=\lf(\sum_{i=1}^3\lf(y^i\ri)^2\ri)^\frac12=\textrm{
constant}\ri\}$. Then
$$|\hnabla
\rho|^2_{\hg}=r^{-2}g^{ij}\frac{y^iy^j}{\rho^2}.$$  Consider the
following metric on $\Sigma_\rho=\{y|\ |y|=\rho\}$:
\begin{equation}\label{embed1-2}
   \begin{split}
     ds_r^2&=r^{-2}\hh_{ij}\\
      & =g_{ij}-r^{-2}\hn_i\hn_j\\
     & = g_{ij}-\frac{y^iy^j}{g^{kl}y^ky^l}.
   \end{split}
\end{equation}
Clearly, the standard metric $h^0_{ij}$ on $\Sigma_\rho$, is
\begin{equation}
ds_0^2=h^0_{ij}=\delta_{ij}-\frac{y^iy^j}{\rho^2}.
\end{equation}
Direct computations show
\begin{equation}\label{embed4}
   ||ds_r^2-ds^2_0||_{3}=O\lf(r^{-\tau}\ri)
\end{equation}
for $\frac12\le \rho\le 2$.  Note that $\Sigma_1$ is the unit
sphere. By  \cite[p.353]{Niren}, we can find an isometric embedding
$\hat X_r$ of $(\mathbb{S}^2,ds_r^2)$ into $\R^3$ such that
\begin{equation} \label{embed4-2}
\|\hat X_r-X_0\|_{2}=O\lf(r^{-\tau}\ri)
\end{equation}
 where $X_0$ is the identity map. Since $X_0\cdot n_0=1$
 where $n_0$ is the unit outward normal of the unit sphere,
 we have $\hat X_r\cdot n_{0,r}=1+O(r^{-\tau})$, where
 $n_{0,r}$ is the unit outward normal of the surface $\hat
 X_r$. If we identify  $S_r$ with metric induced by $g$ with
 $(\mathbb{S}^2,\hh)$, then
$X_r=r\hat X_r$ is an isometric embedding of $S_r$ with metric
induced by $g$. From this  it is easy to see that the first part
of \eqref{lisomreq1} is true.

By \eqref{embed4-2}, we know that $\hat{H}_0-2=O\lf(r^{-\tau}\ri)$,
where $\hat{H}_0$ is the mean curvature of $\hat X_r.$ After
rescaling $r\hat X_r$, we can get the second part of
\eqref{lisomreq1}.
\end{proof}

\begin{lemm} \label{upperbound}
Let $(M,g)$ be an AF manifold with the properties (\ref{daf1}) and
(\ref{daf2}), and let $S_r$ be coordinate spheres.  We have
\begin{equation} \label{LEMMBYIH0}
\int_{S_r} H_0d\Sigma_r=4\pi
r+\frac{\mathcal{A}(r)}{r}+O(r^{1-2\tau}).
\end{equation}
\end{lemm}

\begin{proof}
By Lemma \ref{isomr1}, for $r$ large enough, we can find an
isometric embedding $X_r$ of $S_r$ in $\R^3$ such that $X_r\cdot
n_0=r+O(r^{1-\tau})$. Let $H_0$ be the mean curvature when $S_r$ is
embedded in $\R^3$. By Lemma \ref{2nd}(i),
$$\bar
K=K-\frac1{r^2}=O(r^{-2-\tau}).
$$
By one of the Minkowski integral
 formulae  \cite[Lemma 6.2.9 ]{KLBG},   we have
\begin{equation}\label{BY-ADMH01}
\begin{split}
\int_{S_r}H_0d\Sigma_r&=2\int_{S_r}KX_r\cdot n_0d\Sigma_r\\
&=2\int_{S_r}\lf(\frac1{r^2}+\bar{K}\ri)X_r\cdot n_0d\Sigma_r\\
&=\frac{2}{r^2}\int_{S_r}X_r\cdot n_0d\Sigma_r+2\int_{S_r}\bar{K}
X_r\cdot n_0d\Sigma_r\\
&=\frac{6V_0(r)}{r^2}+2\int_{S_r}\bar{K}\lf(r+O\lf(r^{1-\tau}\ri)\ri)
d\Sigma_r\\
&=\frac{6V_0(r)}{r^2}+2r\int_{S_r}\bar{K}d\Sigma_r+O(r^{1-2\tau})\\
&=\frac{6V_0(r)}{r^2}+2r\int_{S_r}\lf(K-\frac1{r^2}\ri)d\Sigma_r+O(r^{1-2\tau})\\
&=\frac{6V_0(r)}{r^2}+8\pi
r-\frac{2\mathcal{A}(r)}{r}+O(r^{1-2\tau})
\end{split}
\end{equation}
where $V_0(r)$ is the volume of the interior of the surface $X_r$
in $\R^3$. On the other hand, from Lemma \ref{isomr1},
$H_0=\frac2r+H_1$ with $H_1=O\lf(r^{-1-\tau}\ri)$. By another
Minkowski integral formula, we have
\begin{equation}\label{BY-ADMH02}
\begin{split}
2\mathcal{A}(r)&= \int_{S_r}H_0X\cdot n_0d\Sigma_r\\
&=\frac{6V_0(r)}r+\int_{S_r}H_1X\cdot n_0d\Sigma_r\\
&=\frac{6V_0(r)}r+r\int_{S_r}H_1d\Sigma_r+O\lf(r^{2-2\tau}\ri)\\
&=\frac{6V_0(r)}r-2\mathcal{A}(r)+r\int_{S_r}H_0d\Sigma_r+O\lf(r^{2-2\tau}\ri).
\end{split}
\end{equation}
So
\begin{equation}\label{BY-ADMH03}
\int_{S_r}H_0d\Sigma_r=-\frac{6V_0(r)}{r^2}+
\frac{4\mathcal{A}(r)}r +O(r^{1-2\tau}).
\end{equation}
From \eqref{BY-ADMH01} and \eqref{BY-ADMH03}, the lemma follows.
\end{proof}

\begin{proof}[Proof of Theorem \ref{thmlc}] The theorem
follows immediately from Lemmas \ref{meancurvint} and
\ref{upperbound}.
\end{proof}

In Theorem \ref{thmlc}, $S_r$ can be replaced by slightly deformed
spheres. More precisely, we have:
\begin{coro} Same assumptions as in Theorem \ref{thmlc}.
Suppose $\rho$ is a smooth function on $M$ such that
\begin{equation}\label{deformed}
   |\rho-r|+r|\p(\rho-r)|+r^2|\p\p(\rho-r)|+r^3|\p\p\p(\rho-r)|
   =O(r^{\kappa})
\end{equation}
  for
some $0<\kappa<1-\tau$. Then
$$
\lim_{\rho\to\infty}m_{BY}(\Sigma_\rho)=m_{ADM}(M)
$$
where $\Sigma_\rho$ is the level set of the smooth function
$\rho$.
\end{coro}
\begin{proof} Let $y=\frac{\rho}r x=F(x)$. Then one can
show that $y$ is also a coordinates system of $M$ at infinity so
that the metric tensor in this coordinates satisfies the decay
conditions (\ref{daf1}) and (\ref{daf2}). Note that $\Sigma_\rho$
is nothing but the coordinate spheres in the $y$-coordinates.
Hence the corollary follows from the uniqueness of ADM mass by
\cite{BTK86}.
\end{proof}

With the notations as in the proof of Theorem \ref{thmlc}. Let
$V(r)$ be the volume with respect to an AF metric $g$ of the
region inside $S_r$. We can compare $V(r)$ and $V_0(r)$  (Theorem
\ref{vol-comparison-L-in}):
\begin{theo}\label{vol-comparison-L}
With the above notations. Let $(M,g)$ be an asymptotically flat
manifold of order $\tau>\frac12$ with one end.    Then
\begin{equation}\label{volcompL-e1}
    V_0(r)-V(r)=-2m_{ADM}\pi r^2+o(r^2).
\end{equation}
\end{theo}
\begin{proof} Let $m=m_{ADM}$. With the same notations as in the proof of Theorem
\ref{thmlc}, by \eqref{2nd-eq0} and the co-area formula we have
\begin{equation}\label{volcompL-e2}
   \begin{split}
 V'(r)&=\int_{S_r}\frac1{|\nabla r|}d\Sigma_r\\
 &=\mathcal{A}(r)+\frac12\int_{S_r}\frac{\sigma_{ij}x^ix^j}{r^2}d\Sigma_r+O(r^{2-2\tau}).
\end{split}
\end{equation}
Here and  below $'$ is the derivative with respect to  $r$. On the
other hand, by \eqref{meancurvint-e1} and \eqref{meancurvint-e2} we
have
\begin{equation}\label{volcompL-e3}
   \begin{split}
 \mathcal{A}'(r)&=\frac{\mathcal{A}(r)}r+4\pi r-8\pi m+
 \int_{S_r}\frac{\sigma_{ij} x^ix^j}{r^3} d\Sigma_r+o(1).
\end{split}
\end{equation}
Eliminating the term $\int_{S_r}\frac{\sigma_{ij} x^ix^j}{r^3}
d\Sigma_r$ from \eqref{volcompL-e2} and \eqref{volcompL-e3} we
have
\begin{equation}\label{volcompL-e4}
 \mathcal{A}'(r)=\frac{\mathcal{A}(r)}r+4\pi r-8\pi
 m+\frac 1r\lf(2V'(r)-2\mathcal{A}(r)\ri)+o(1).
\end{equation}

  Hence
$$
\lf(r\mathcal{A}(r)\ri)'=4\pi r^2-8\pi mr+2V'(r)+o(r)
$$
and
\begin{equation}\label{volcomp-ne3}
V(r)=\frac12 r\mathcal{A}(r)-\frac{2\pi r^3}3+2\pi mr^2+o(r^2).
\end{equation}
On the other hand, by \eqref{BY-ADMH01} and \eqref{BY-ADMH02}, we
have
\begin{equation}\label{volcomp-ne4}
    V_0(r)=\frac12 r\mathcal{A}(r)-\frac{2\pi
    r^3}3+O(r^{3-2\tau}).
\end{equation}
Hence
$$
V_0(r)-V(r)=-2\pi mr^2+o(r^2)
$$
because $\tau>\frac12$.

\end{proof}

Combine this with Positive  Mass Theorem, we have the following

\begin{coro}\label{colsafvc1} With above notations, let $(M,g)$ be an AF manifold of
order $\tau >\frac12$. If the scalar curvature is nonnegative, then
$$
\lim_{r\rightarrow +\infty} \frac{V(r)-V_0 (r)}{r^2}\ge 0,
$$
and equality holds if and only if $(M,g)$ is isometric to
$\mathbb{R}^3$.
\end{coro}

From the proof of Theorem \ref{vol-comparison-L}, Miao \cite{Miao}
is able to obtain the following. Thanks to Pengzi Miao, we
include the result and the proof here.
\begin{coro}\label{iso-mass-cor} In an AF manifold $M$, the ADM mass and the
isoperimetric mass  introduced by Huisken \cite{Huisken} are
equal.
\end{coro}
\begin{proof} Recall that the isoperimetric mass of $M$ is defined
as
$$
m_{ISO}=\limsup_{r\to\infty}\frac2{\mathcal{A}(r)}\lf(V(r)-\frac1{6\pi^\frac12}\mathcal{A}^\frac32(r)\ri).
$$
Now by \eqref{volcomp-ne3}
\begin{equation}\label{isomass-e1}
\begin{split}
    \frac2{\mathcal{A}(r)}&\lf(V(r)-\frac1{6\pi^\frac12}
    \mathcal{A}^\frac32(r)\ri)\\
     &
     =r+\frac1{\mathcal{A}(r)}\lf(4\pi m r^2-\frac{4\pi r^3}3\ri)-\frac1{3\pi^\frac12}\mathcal{A}^\frac12(r)+o(1)\\
        &=r+\lf(m-\frac r3\ri)\lf(1-\mathcal{I}+O(r^{-2\tau})\ri)-\frac
        {2r}3\lf(1+\frac12 \mathcal{I}+O(r^{-2\tau})\ri)+o(1)\\
        &=m+o(1).
\end{split}
\end{equation}
where
$$
\mathcal{I}=\frac{1}{8\pi r^2} \int_{S_r} h^{ij}\sigma_{ij}
d\Sigma_r=O(r^{-\tau})
$$
so that
$$
\mathcal{A}(r)=4\pi r^2\lf(1+\mathcal{I}+O(r^{-2\tau})\ri),
$$
see Lemma \ref{2nd}(ii). From this the result follows.
\end{proof}

\section{small-sphere  limit}

In this section, we will first study the small-sphere limit of the
Brown-York mass of  geodesic spheres up to order $r^5$ where $r$
is the geodesic distance from a fixed point. Let $(N^3,g)$ be a
three dimensional manifold and let $p\in N$. Let $\{x^i\}$ be the
normal coordinates near $p$. By \cite[Chapter 5]{SYL94}, we have
the following expansion of $g$ near $p$:

\begin{lemm}\label{metric-expansion-1}
For any point $x$ close to $p$, the metric components of $g$ in
the normal coordinates can be expressed as
\begin{equation} \label{ssbymproofthm1}
\begin{split}
g_{ij}(x) & = \delta_{ij}+\frac{1}{3}R_{iklj}(p)x^kx^l+
\frac{1}{6}R_{iklj;m}(p)x^kx^lx^m\\
&
+\lf(\frac{1}{20}R_{iklj;mn}(p)+\frac{2}{45}R_{ikls}(p)R_{jmns}(p)\ri)x^kx^lx^mx^n+O\lf(r^5\ri)\\
\end{split}
\end{equation}
and
\begin{equation}\label{4area-s1}
 \begin{split}
    g&=\det(g_{ij})\\
    &=1-\frac13 R_{ij}(p)x^ix^j-\frac16
    R_{ij;k}(p)x^ix^jx^k\\
    &\quad-\lf(\frac1{20}R_{ij;kl}(p)+\frac1{90}R_{hijm}(p)R_{hklm}(p)
    -\frac1{18}R_{ij}(p)R_{kl}(p)\ri)x^ix^jx^kx^l+O(r^5),
\end{split}
\end{equation}
 where $r$ is the geodesic distance from $p$, $R_{ijkl}$ is the Riemannian curvature tensor, $R_{ij}$ is the Ricci curvature
 and $R$ is the scalar curvature with respect
to the metric $g$, and $R_{iklj;m}$ is the covariant derivative of
$R_{ijkl}$ etc.
\end{lemm}
In our notations, the sectional curvature is nonnegative if
$R_{ijij}\ge 0$. In the following, we always assume that the
normal coordinates are chosen so that at $p$ the Ricci curvature
is of the form $R_{ij}=\lambda_i\delta_{ij}$ where
$\lambda_1,\lambda_2,\lambda_3$ are the eigenvalues of $R_{ij}$.

\begin{lemm} \label{expansionofarea}
Let $\mathcal{A}(r)$ be the area of geodesic sphere
$S_r=\{|x|=r\}$ with radius $r$ in $(N,g)$ with center at $p$,
then:
\begin{equation}
\mathcal{A}(r)=4\pi r^2 +A_4 + A_6 +O(r^7),
\end{equation}
where
\begin{equation}
A_4=-\frac{2\pi r^4}{9}R,\quad A_6=\frac{\pi r^6}{675}\lf (4R^2
-2|Ric|^2-9\Delta R\ri)
\end{equation}
where, $\Delta$ is the Laplacian operator with respect to metric
$g$ and $|Ric|$ is the norm of the Ricci tensor. Here all the
terms involving curvature are evaluated at $p$.
\end{lemm}

\begin{proof}

By \eqref{4area-s1}
$$
\sqrt g=1+\frac12(b_2+b_3+b_4)-\frac18 b_2^2+O(r^5)
$$
where

\begin{equation}\label{4area-s1-1}
\begin{split}
 b_2&=-\frac{1}{3}R_{ij}x^ix^j\\
 b_3&=-\frac16
    R_{ij,k}x^ix^jx^k\\
     b_4&=-\lf(\frac1{20}R_{ij,kl}+\frac1{90}R_{hijm}R_{hklm}-\frac1{18}R_{ij}R_{kl}\ri)x^ix^jx^kx^l.
\end{split}
\end{equation}
Hence
$$
V(r)=\int_{B_r}\sqrt g dv_0
$$
where $dv_0$ is the volume element with respect to Euclidean
metric and $B_r=\{x|\ |x|<r\}$. Since $|\frac{\p}{\p r}|=1$ in $g$
metric,
\begin{equation}\label{4area-s2}
\begin{split}
\mathcal{A}(r)&=V'(r)\\
&=4\pi r^2+\int_{S_r}\lf[\frac12(b_2+b_3+b_4)-\frac18
b_2^2+O(r^5)\ri]d\Sigma_r^0\\
&=4\pi
r^2+\frac12\int_{S_r}b_2d\Sigma_r^0+\frac12\int_{S_r}\lf(b_4-\frac14
b_2^2\ri)d\Sigma_r^0+O(r^7)
 \end{split}
 \end{equation}
 where $d\Sigma_r^0$ is the area element of $S_r$ with
 respect to the Euclidean metric.
 Since  $\int_{S_r}\lf(x^i\ri)^2d\Sigma_r^0=\frac{4}{3}\pi
 r^4$, by \eqref{4area-s1-1} and the fact that
 $R_{ij}x^ix^j=\sum_{i=1}^3\lambda_i(x^i)^2$,
\begin{equation}\label{4area-s3-1}
\frac{1}{2}\int_{S_r}b_2d\Sigma^0_r=-\frac{2\pi r^4}9R.
\end{equation}
Noting that
\begin{equation}\label{4area-s4}
\begin{split}
b_4&=-\frac1{20}R_{ij;kl}x^ix^jx^kx^l-\frac1{90}\sum_{h,m}\lf(\sum_{ij}
R_{hijm}x^ix^j\ri)^2+\frac1{18}\lf(\sum_{ij}R_{ij}x^ix^j\ri)^2\\
&=-\frac1{20}R_{ij;kl}x^ix^jx^kx^l-\frac1{90}\sum_{i,j}\lf(\sum_{k,l}
R_{iklj}x^kx^l\ri)^2+\frac12b_2^2.
\end{split}
\end{equation}
Let us first  compute
$\int_{S_r}\lf(R_{ij}x^ix^j\ri)^2d\Sigma_r^0.$ By symmetry
\begin{equation}\label{4basicint1}
\int_{S_r}\lf(x^i\ri)^4d\Sigma_r^0=\frac{4}{5}\pi r^6, \textrm{
for } i=1,2,3
\end{equation}
and for $i\neq j$,
\begin{equation}\label{4basicint2}
\int_{S_r}\lf(x^ix^j\ri)^2d\Sigma_r^0=\frac{4}{15}\pi r^6.
\end{equation}
We have
\begin{equation}\label{4Areafricci11}
\begin{split}
&\int_{S_r}\lf(R_{ij}x^ix^j\ri)^2d\Sigma_r^0\\
&=\int_{S_r}\lf(\sum_i\lambda_i(x^i)^2\ri)^2d\Sigma_r^0\\
&=\frac{4}{5}\pi r^6\sum_i\lambda_i^2+\frac{4}{15}\pi
r^6\sum_{i\neq
j}\lambda_i\lambda_j\\
&=\frac{4}{15}\pi  r^6\lf(R^2+2|Ric|^2\ri).
\end{split}
\end{equation}
Since the $\dim N=3$,  by  \cite[ p.276]{Hamilton}, at $p$
\begin{equation}\label{4curv-s1}
\begin{split}
    R_{ijkl}
    &=
    g_{ik}R_{jl}-g_{il}R_{jk}-g_{jk}R_{il}+g_{jl}R_{ik}-\frac{1}{2}R\lf(g_{ik}g_{jl}-g_{il}g_{jk}\ri)\\
    &=\delta_{ik}R_{jl}-\delta_{il}R_{jk}-\delta_{jk}R_{il}+\delta_{jl}R_{ik}
    -\frac{1}{2}R\lf(\delta_{ik}\delta_{jl}-\delta_{il}\delta_{jk}\ri)
\end{split}
\end{equation}
and hence on $S_r$:
\begin{equation}\label{4curv-s2}
\begin{split}
    R_{ijkl}x^jx^k
    &=\lf(\lambda_i+\lambda_l-\frac R2\ri)x^ix^l-\delta_{il}\lf(\sum_{k}\lambda_k
    (x^k)^2+\lf(\lambda_i-\frac R2\ri) r^2\ri).
    \end{split}
\end{equation}
Using \eqref{4basicint2}, we have
\begin{equation}\label{4curv-s3}
\begin{split}
  \int_{S_r}\sum_{i\neq l} \lf(\sum_{j,k} R_{ijkl}x^jx^k\ri)^2d\Sigma_r^0&
  =\int_{S_r}\sum_{i\neq
  l}\lf(\lambda_l+\lambda_i-\frac R2\ri)^2\lf(x^ix^l\ri)^2d\Sigma_r^0\\
  &=\frac{4}{15}\pi r^6\sum_{i\neq l}\lf(\lambda_l+\lambda_i-\frac R2\ri)^2\\
  &=\frac{4}{15}\pi r^6\lf(2|Ric|^2-\frac {R^2}2 \ri).
 \end{split}
\end{equation}
Clearly, by \eqref{4curv-s2}, \eqref{4basicint1} and
\eqref{4basicint2}, we have
\begin{equation}\label{4curv-s4-1}
\begin{split}
  & \int_{S_r} \lf(\sum_{jk}R_{1jk1}x^jx^k\ri)^2d\Sigma^0_r\\
   &= \int_{S_r}\lf(\lf(2\lambda_1-\frac R2\ri)(x^1)^2-
    \sum_{k}\lambda_k(x^k)^2-\lf(\lambda_1-\frac R2\ri)r^2\ri)^2d\Sigma^0_r\\
    &=\pi r^6\lf(-\frac{4}{15}\lambda_1^2+\frac{8}{15}R\lambda_1-\frac{4}{15}R^2+\frac{8}{15}|Ric|^2\ri). \end{split}
\end{equation}
We have similar formula for the case $i=l=2$ or 3. So
\begin{equation}\label{4curv-s5}
\begin{split}
  \int_{S_r}\sum_{i=l} \lf(\sum_{j,k} R_{ijkl}x^jx^k\ri)^2d\Sigma_r^0&
  =\pi r^6\lf(\frac{20}{15}|Ric|^2-\frac{4}{15}R^2\ri).
 \end{split}
\end{equation}
By \eqref{4curv-s3} and \eqref{4curv-s5} we have
\begin{equation}\label{4curv-s6}
   \int_{S_r}\sum_{i,l}\lf(\sum_{j,k}R_{ijkl}x^jx^k\ri)^2d\Sigma_r^0
   =\frac{1}{15}\lf(28|Ric|^2-6R^2\ri)\pi r^6.
\end{equation}

Finally, let us compute
$\int_{S_r}\sum_{i,j,k,l}R_{ij;kl}x^ix^jx^kx^ld\Sigma_r^0$. By
symmetry $\int_{S_r} R_{ij;kl}x^ix^jx^kx^ld\Sigma_r^0=0$ unless
$x^ix^jx^kx^l$ is of the form $(x^m)^4$, or $(x^m)^2(x^n)^2$ with
$m\neq n$.

\begin{equation}\label{4Areafriccdire1}
\begin{split}
&\sum_{j,k,l}\int_{S_r}R_{1j;kl}x^1x^jx^kx^ld\Sigma_r^0\\
&=\int_{S_r}R_{11;11}(x^1)^4d\Sigma_r^0+\int_{S_r}R_{11;22}(x^1x^2)^2d\Sigma_r^0
+\int_{S_r}R_{11;33}(x^1x^3)^2d\Sigma_r^0\\
&\quad
+\int_{S_r}(R_{12;12}+R_{12;21})(x^1x^2)^2d\Sigma_r^0+\int_{S_r}(R_{13;13}+R_{13;31})(x^1x^3)^2d\Sigma_r^0\\
&=\frac{4}{5}\pi r^6R_{11;11}+\frac{4}{15}\pi
r^6(R_{12;12}+R_{12;21}+R_{13;13}+R_{13;31}+R_{11;22}+R_{11;33}).
\end{split}
\end{equation}
Similarly, one can prove that
\begin{equation}\label{4Areafriccdire2}
\begin{split}
&\sum_{j,k,l}\int_{S_r}R_{2j,kl}x^2x^jx^kx^ld\Sigma_r^0\\
&=\frac{4}{5}\pi r^6R_{22;22}+\frac{4}{15}\pi
r^6(R_{21;21}+R_{21;12}+R_{23;23}+R_{23;32}+R_{22;11}+R_{22;33}),
\end{split}
\end{equation}
and
\begin{equation}\label{4Areafriccdire3}
\begin{split}
&\sum_{j,k,l}\int_{S_r}R_{3j,kl}x^3x^jx^kx^ld\Sigma_r^0\\
&=\frac{4}{5}\pi r^6R_{33;33}+\frac{4}{15}\pi
r^6(R_{31;31}+R_{31;13}+R_{32;32}+R_{32;23}+R_{33;11}+R_{33;22}).
\end{split}
\end{equation}
Hence
\begin{equation}\label{4Areafriccdire4}
\begin{split}
&\sum_{i,j,k,l}\int_{S_r}R_{ij;kl}x^ix^jx^kx^ld\Sigma_r^0\\
&=\frac{4}{15}\pi r^6\sum_{i,j}\lf(R_{ii;jj}+2R_{ij;ij}\ri).
\end{split}
\end{equation}
By the second Bianchi identity, we see that

$$\sum_i R_{ij;ij}=\frac12 R_{;jj}.$$
Therefore, we have

\begin{equation}\label{4Areafriccdire5}
\sum_{i,j,k,l}\int_{S_r}R_{ij,kl}x^ix^jx^kx^ld\Sigma_r^0 =
\frac{8\pi r^6}{15}\Delta R(p). \end{equation}
 The lemma follows from  \eqref{4area-s1-1},
 \eqref{4area-s2}, \eqref{4area-s3-1}, \eqref{4area-s4},
 \eqref{4Areafricci11}, \eqref{4curv-s4-1}, \eqref{4curv-s6}, and
  \eqref{4Areafriccdire5}.
 \end{proof}
\begin{coro}\label{expansionofarea-cor}
With the notations and assumptions as in Lemma
\ref{expansionofarea}, let $H$ be the mean curvature of $S_r$ with
respect to $g$, then
\begin{equation}\label{meancurvint-small}
  \int_{S_r}Hd\Sigma_r=8\pi r+\frac{4A_4+6A_6}{r}+O(r^6).
\end{equation}
\end{coro}
\begin{proof}  By the fact that $|\nabla r|=1$, we have
$$
\int_{S_r}Hd\Sigma_r=\frac{d}{dr}\mathcal{A}(r).
$$
The corollary then follows from Lemma \ref{expansionofarea}.
\end{proof}

By \cite{Niren}, and the fact that for $r$ small $(S_r,g|_{S_r})$
has positive Gauss curvature, one can isometrically embed
$(S_r,g|_{S_r})$ in $\R^3$.
\begin{lemm}\label{normalexpansion}
For $r$ small enough, there is an    isometric embedding $Z$  of
geodesic sphere $S_r$ into $\mathbb{R}^3$ such that
\begin{equation}
Z\cdot n=r+\frac{r^3}{6}\lf(\frac R2-2R_{ij}\frac{x^ix^j}{r^2} \ri)
+ O(r^4),
\end{equation}
where $n$ be the outward unit normal vector of $Z(S_r)$ in
$\mathbb{R}^3$ and `$\cdot$' is the inner product in $\R^3$.
\end{lemm}

\begin{proof}
For $r>0$, we define a map $x=ry$ and pull back the metric $g$ to
the $y$ space and let $h$ be the   metric $r^{-2}g$ induced on the
unit sphere $ \mathbb{S}^2$ in the $y$ space. As in the proof of
Lemma \ref{isomr1}, in order to prove the lemma  it is sufficient
to prove that for $r$ small, we can find an isometric embedding
$Z_r$ of $(\mathbb{S}^2, h)$ in $\R^3$ such that
\begin{equation}
Z_r\cdot n_r=1+\frac{r^2}{6}\lf(\frac R2-2\sum
\lambda_k(y^k)^2\ri) + O(r^3),
\end{equation}
where $n_r$ is the unit outward normal of $Z_r(\mathbb{S}^2)$.

 Let $\th$ be the
induced metric of $r^{-2}\tg$ on the unit sphere $\mathbb{S}^2$,
where
\begin{equation}\label{pullback-1}
    \tg_{ij}=\tg\lf(\frac{\p}{\p y^i},\frac{\p}{\p y^j}\ri)=r^2
    \lf(\delta_{ij}+\frac{r^2}{3}R_{iklj}y^ky^l\ri).\nonumber
\end{equation}

Let $\hat h$ be the metric on $\mathbb{S}^2$ induced by the pull
back of the Euclidean metric given by the embedding $\hat
Z=(z^1,z^2,z^3)$ in $\R^3$ where
\begin{equation}
    \begin{split}
     z^1 & =y^1
     \lf(1+\frac {r^2}6\lf(\frac R2-\lambda_1-
     \sum_{i}\lambda_{i}(y^i)^2\ri)\ri) \\
      z^2  &= y^2
     \lf(1+\frac {r^2}6\lf(\frac R2-\lambda_2-
     \sum_{i}\lambda_{i}(y^i)^2\ri)\ri)\\
      z^3&= y^3
     \lf(1+\frac {r^2}6\lf(\frac R2-\lambda_3-
     \sum_{i}\lambda_{i}(y^i)^2\ri)\ri).
    \end{split}\nonumber
\end{equation}
We  claim that
\begin{equation}\label{approx-e1}
    ||h-\th  ||_{C^3}+||\hat h-\th   ||_{C^3}=O(r^3)
\end{equation}
where the norm is computed with respect to the standard metric.
Suppose the claim is true, then by \cite{Niren}, we can conclude
that there are isometric embeddings $Z_r$, $\tilde Z$ and $\hat Z$
for $(\mathbb{S}^2, h)$, $(\mathbb{S}^2, \th  )$ and
 $(\mathbb{S}^2, \hat h)$ respectively such that
\begin{equation}\label{approx-e2}
    ||Z_r\cdot n_r-\hat Z \cdot \hat n  ||=O(r^3)
\end{equation}
where $\hat n$ is the unit outward normal of $\hat
Z(\mathbb{S}^2)$. Then we can prove the lemma by computing $\hat Z
\cdot \hat n $.

Let us first prove the claim and then  compute $\hat Z \cdot \hat
n $. It is easy to see that $||h-\th  ||_{C^3}=O(r^3)$ by the
expression of  $g$ in Lemma \ref{metric-expansion-1} and the
definition of  $\tg$.

To find $\th$, by \eqref{4curv-s2}, on the unit sphere of the $y$
space:
 $r^{-2}\wt g_{ij}=\delta_{ij}+\wt \sigma_{ij}$, and
$\lambda_{ij}=-(\lambda_i+\lambda_j)+\frac R2$, then
\begin{equation}\label{metric-s2}
\begin{split}
    \wt\sigma_{11}&=
    \frac{r^2}3(\lambda_{12}(y^2)^2+\lambda_{13}(y^3)^2)\\
    \wt\sigma_{22}&=\frac{r^2}3(\lambda_{12}(y^1)^2+\lambda_{23}(y^3)^2)\\
 \wt\sigma_{33}&=\frac{r^2}3(\lambda_{13}(y^1)^2+\lambda_{23}(y^2)^2)\\
\wt\sigma_{ij}&=-\frac{r^2}3\lambda_{ij}y^i y^j, ~\text{$i\neq
j$}.
\end{split}\nonumber
\end{equation}
In the above last equation  the repeated indices is not taken
summation.

Let $e_1=\p_\theta=a_i\p_i$, $e_2=(\sin
\theta)^{-1}\p_\phi=b_i\p_i$. Then
\begin{equation}\label{metric-s3}
    a_1=\cos\theta\cos\phi, a_2=\cos\theta\sin\phi,
    a_3=-\sin\theta; b_1=-\sin\phi, b_2=\cos\phi, b_3=0.\nonumber
\end{equation}
Note that $$y^1=\sin\theta\cos\phi, y^2=\sin\theta\sin\phi,
y^3=\cos\theta. $$ Hence in the basis $\{e_1, e_2\}$, $\th$ is
given by

\begin{equation}\label{embedding-s1}
\begin{split}
\th_{11}&=1+\frac{r^2}3
\lf(\lambda_{23}\sin^2\phi+\lambda_{31}\cos^2\phi\ri),\\
\th_{12}&=
  \frac{r^2}3\lf(-\lambda_{13}+\lambda_{23}\ri)\cos\theta\cos\phi\sin\phi,\\
  \th_{22}&=1+ \frac{r^2}3\lf(\lambda_{12}\sin^2\theta+
\lambda_{13}\cos^2\theta\sin^2\phi+\lambda_{23}\cos^2\theta\cos^2\phi\ri).
\end{split}
\end{equation}
Next we want to compute $\hat h$.
\begin{equation}
    \begin{split}
     (z^1)_\theta& =(y^1)_\theta
     \lf(1+\frac {r^2}6\lf(\frac R2-\lambda_1-
     \sum_{i}\lambda_{i}(y^i)^2\ri)\ri)-\frac {r^2}6
     y^1\lf(\sum_{i}\lambda_{i}(y^i)^2\ri)_\theta\\
(z^2)_\theta&= (y^2)_\theta
     \lf(1+\frac {r^2}6\lf(\frac R2-\lambda_2-
     \sum_{i}\lambda_{i}(y^i)^2\ri)\ri)-\frac {r^2}6
     y^2\lf(\sum_{i}\lambda_{i}(y^i)^2\ri)_\theta\\
     (z^3)_\theta&=(y^3)_\theta
     \lf(1+\frac {r^2}6\lf(\frac R2-\lambda_3-
     \sum_{i}\lambda_{i}(y^i)^2\ri)\ri)-\frac {r^2}6
     y^3\lf(\sum_{i}\lambda_{i}(y^i)^2\ri)_\theta
    \end{split}\nonumber
    \end{equation}
    and
 \begin{equation}
    \begin{split}
     (z^1)_\phi& =(y^1)_\phi
     \lf(1+\frac {r^2}6\lf(\frac R2-\lambda_1-
     \sum_{i}\lambda_{i}(y^i)^2\ri)\ri)-\frac {r^2}6
     y^1\lf(\sum_{i}\lambda_{i}(y^i)^2\ri)_\phi\\
(z^2)_\phi&= (y^2)_\phi
     \lf(1+\frac {r^2}6\lf(\frac R2-\lambda_2-
     \sum_{i}\lambda_{i}(y^i)^2\ri)\ri)-\frac {r^2}6
     y^2\lf(\sum_{i}\lambda_{i}(y^i)^2\ri)_\phi\\
     (z^3)_\phi&=(y^3)_\phi
     \lf(1+\frac {r^2}6\lf(\frac R2-\lambda_3-
     \sum_{i}\lambda_{i}(y^i)^2\ri)\ri)-\frac {r^2}6
     y^3\lf(\sum_{i}\lambda_{i}(y^i)^2\ri)_\phi.
    \end{split}\nonumber
    \end{equation}
       Hence
    \begin{equation}\begin{split}
\hat Z_\theta\cdot\hat Z_\theta&=(z^1)_\theta^2+(z^2)_\theta^2+(z^3)_\theta^2\\
&=\sum_i(y^i)_\theta^2
     \lf(1+\frac {r^2}3\lf(\frac R2-\lambda_i-
     \sum_{k}\lambda_{k}(y^k)^2\ri)\ri)\\
     &-\frac {r^2}3
     \sum_iy^i(y^i)_\theta\lf(\sum_{i}\lambda_{i}(y^i)^2\ri)_\theta+O(r^4)\\
     &=1+\frac{r^2}3\lf[\frac R2-\sum_i
     \lambda_i\lf((y^i)_\theta^2+(y^i)^2)\ri)\ri]+O(r^4)\\
     &=1+\frac{r^2}3\lf(\frac R2-
     \lambda_1\cos^2\phi-\lambda_2\sin^2\phi-\lambda_3
      \ri)+O(r^4)\\
      &=1+\frac{r^2}3\lf(
     \lambda_{13}\cos^2\phi+\lambda_{23}\sin^2\phi
      \ri)+O(r^4),
\end{split}\nonumber
\end{equation}

 \begin{equation}\begin{split}
\hat Z_\phi&\cdot \hat Z_\phi\\
&=(z^1)_\phi^2+(z^2)_\phi^2+(z^3)_\phi^2\\
&=\sum_i(y_i)_\phi^2
     \lf(1+\frac {r^2}3\lf(\frac R2-\lambda_i-
     \sum_{k}\lambda_{k}y_k^2\ri)\ri)- \frac {r^2}3
     \sum_iy_i(y_i)_\phi\lf(\sum_{i}\lambda_{i}y_i^2\ri)_\phi+O(r^4)\\
     &=\sin^2\theta+\frac{r^2}3\lf[\frac R2\sin^2\theta-\sum_i
     \lambda_i\lf((y_i)_\phi^2+y_i^2\sin^2\theta \ri)\ri]+O(r^4)\\
     &=\sin^2\theta\bigg[1+\frac{r^2}3\Big(\frac R2-
     \lambda_1\lf(\sin^2\phi+\cos^2\phi\sin^2\theta\ri)-\lambda_2\lf(
     \cos^2\phi+\sin^2\phi\sin^2\theta\ri)\\
     &-\lambda_3\cos^2\theta\Big)\bigg]+O(r^4)\\
     &=\sin^2\theta
     \lf[1+\frac{r^2}3\lf( \lambda_{12}\sin^2\theta+\lambda_{13}\cos^2\theta\sin^2\phi
     +\lambda_{23}\cos^2\theta\cos^2\phi\ri)\ri]+O(r^4),
\end{split}\nonumber
\end{equation}
and
\begin{equation}\begin{split}
\hat Z_\theta\cdot\hat Z_\phi&=(z^1)_\theta (z^1)_\phi
+(z^2)_\theta
(z^2)_\phi +(z^3)_\theta (z^3)_\phi\\
&=\sum_i(y^i)_\theta (y^i)_\phi+\frac{r^2}3\sum_i (y^i)_\theta
(y^i)_\phi\lf(\frac R2-\lambda_i-\sum_k\lambda_k (y^k)^2\ri)\\
&-\frac{r^2}6\lf(\sum_{i}y^i(y^i)_\theta\ri)
\lf(\sum_{i}\lambda_{i}(y^i)^2\ri)_\phi
-\frac{r^2}6\lf(\sum_{i}y^i(y^i)_\phi\ri)\lf(\sum_{i}\lambda_{i}(y^i)^2\ri)_\theta+O(r^4)\\
&=-\frac{r^2}3\sum_i (y^i)_\theta
(y^i)_\phi\lambda_i+O(r^4)\\
&=\frac{r^2}3\lf(-\lambda_{13}+\lambda_{23}\ri)\sin\theta\cos\theta\cos\phi\sin\phi+O(r^4).
\end{split}\nonumber
\end{equation}

Thus, we see that
\begin{equation}\label{difference of metric1}
\|\th - \hat h\|_{C^3} = O(r^4).
\end{equation}
This completes the proof of the claim. Next we want to compute
$\hat Z\cdot \hat n$.

Let $A=(y^1,y^2,y^3)$, $B=A_\theta$, $C=\frac1{\sin\theta}A_\phi$,
$\bar A=(\bar A_1, \bar A_2,\bar A_3)$, $\bar B=(\bar B_1, \bar
B_2,\bar B_3)$, and $\bar C=(\bar C_1, \bar C_2,\bar C_3)$ where
\begin{equation}
 \begin{split}
    \bar A_i  &= \frac {r^2y^i}6\lf(\frac R2-\lambda_i-
     \sum_{i}\lambda_{k}(y^k)^2\ri)\\
      \bar B_i  &=
      \frac {(y^i)_\theta r^2}6\lf(\frac R2-\lambda_i-
     \sum_{k}\lambda_{k}(y^k)^2\ri) -\frac {r^2}6
     y^i\lf(\sum_{k}\lambda_{k}(y^k)^2\ri)_\theta\\
     \bar C_i&=
      \frac {(y^i)_\phi r^2}{6\sin\theta}\lf(\frac R2-\lambda_i-
     \sum_{k}\lambda_{k}(y^k)^2\ri)-\frac {r^2}{6\sin\theta}
     y^i\lf(\sum_{k}\lambda_{k}(y^k)^2\ri)_\phi.
\end{split}\nonumber
\end{equation}
Note that $A, B, C$ are orthonormal and positively oriented in
$\R^3$ for $A\in  \mathbb{S}^2$. Let $e_1=\p_\theta$ and
$e_2=\frac1{\sin\theta}\p_\phi$ as before. Then
\begin{equation}
\begin{split}
    \hat Z\cdot \hat Z_1\wedge \hat Z_2 & =(A+\bar A)\cdot (B+\bar B)\wedge(C+\bar C)\\
        &=A\cdot B\wedge C+A\cdot B\wedge \bar C+ A\cdot \bar B\wedge
        C+\bar A\cdot B\wedge   C+O(r^4)\\
        &=1+C\cdot \bar C+B\cdot \bar B+A\cdot \bar A+O(r^4).
\end{split}\nonumber
\end{equation}
Now
\begin{equation}
\begin{split}
A\cdot \bar A&=\frac{r^2}6\lf(\frac
R2-2\sum_{k}\lambda_k(y^k)^2\ri).
\end{split}\nonumber
\end{equation}
\begin{equation}
\begin{split}
B\cdot \bar B&=\frac{r^2}6\lf(\frac R2-
\sum_{k}\lambda_k\lf((y^k)_\theta^2+(y^k)^2\ri)\ri)\\
&=\frac{r^2}6\lf(\frac R2-
 \lf(\lambda_1\cos^2\phi+\lambda_2\sin^2\phi+\lambda_3 \ri)\ri).
\end{split}\nonumber
\end{equation}
\begin{equation}
\begin{split}
C\cdot \bar C&=\frac{r^2}{6\sin^2\theta}\lf(\frac R2\sin^2\theta-
\sum_{k}\lambda_k\lf((y^k)_\phi^2+(y^k)^2\sin^2\theta\ri)\ri)\\
&=\frac{r^2}6\lf(\frac R2-
 \lf(\lambda_1\lf(\sin^2\phi+\sin^2\theta\cos^2\phi\ri)
 + \lambda_2\lf(\cos^2\phi+\sin^2\theta\sin^2\phi\ri)+\lambda_3\cos^2\theta \ri)\ri).
\end{split}\nonumber
\end{equation}
So
\begin{equation}\label{Zdotn-e1}
\begin{split}
 \hat Z\cdot \hat Z_1\wedge \hat Z_2 &=1+\frac{r^2}{6}\lf(\frac R2-3\sum
\lambda_k(y^k)^2\ri)+O\lf(r^4\ri).
\end{split}
\end{equation}

Noting that
\begin{equation}
\begin{split}
&|\hat Z_1\wedge \hat Z_2 |^2\\
&=1+\frac{r^2}{3}[\lf(\lambda_{13}(\cos^2\phi
+\cos^2\theta\sin^2\phi)+\lambda_{12}\sin^2\theta
+\lambda_{23}(\sin^2\phi+\cos^2\theta\cos^2\phi)\ri)]+O\lf(r^4\ri)\\
&=1-\frac{r^2}{3}\sum\lambda_k(y^k)^2+O\lf(r^4\ri).
\end{split}\nonumber
\end{equation}
we have
\begin{equation}\label{Zdotn-e2}
 |\hat Z_1\wedge\hat Z_2
|^{-1}=1+\frac{r^2}{6}\sum\lambda_k(y^k)^2+O\lf(r^4\ri).
\end{equation}

Combining \eqref{Zdotn-e1} and \eqref{Zdotn-e2}
\begin{equation}
\begin{split}
\hat Z\cdot \hat n=\frac{\hat Z\cdot \hat Z_1\wedge \hat Z_2
}{|\hat Z_1\wedge \hat Z_2 |}=1+\frac{r^2}{6}\lf(\frac R2-2\sum
\lambda_k(y^k)^2\ri)+O\lf(r^4\ri).
\end{split}
\end{equation}
This completes the proof of the lemma.
\end{proof}

\begin{lemm}\label{Gauss-mean} Let $K$ and $H$ be the Gauss
curvature  and the mean curvature of $S_r$ in $g$ and $H_0$ be the
mean curvature of $(S_r,g|_{S_r})$ when embedded in $\R^3$. Then
\begin{equation}\label{Gauss-mean-e1}
   K=\frac1 {r^2}+\frac R2-\frac 43R_{ij}\frac{x^ix^j}{r^2}+O(r),
\end{equation}
\begin{equation}\label{Gauss-mean-e2}
  H=\frac2 {r}-\frac13  R_{ij}\frac{x^ix^j}r +O(r^2),
\end{equation}
and
\begin{equation}\label{Gauss-mean-e3}
  H_0=\frac2 {r}+ r\lf(\frac R2-\frac 43  R_{ij}\frac{x^ix^j}{r^2}\ri)+O(r^2).
\end{equation}
\end{lemm}
\begin{proof}
We continue to use the normal coordinates as in Lemma
\ref{metric-expansion-1}. Then $n=\frac{\p }{\p r}$ is the outward
normal of $S_r$. Let $h_{ij}=g_{ij}-n_in_j$,  be the induced
metric on $S_r$ with $n_i=\frac{x^i}r$. By Lemma
\ref{metric-expansion-1}, the Christoffel symbols are given by:
\begin{equation}\label{connection-small}
    \Gamma_{ij}^k=\frac13\lf(R_{kimj}+R_{kjmi}\ri)x^m+O(r^2).
\end{equation}
where the curvature are evaluated at $p$. Since $\nabla_nn=0$,
  the second fundamental form $A$ in these coordinates is
given by
\begin{equation}\label{2nd-small-e1}
    \begin{split}
A_{ij}&=n_{j;i}\\
&=\frac{\p n_j}{\p x^i}-\Gamma_{ij}^kn_k\\
&=\frac{\delta_{ij}}r-\frac{x^ix^j}{r^3}-\frac23R_{kimj}\frac{x^kx^m}r+O(r^2).
\end{split}
\end{equation}
Let $e_1$, $e_2$ be orthonormal frame with respect to the
Euclidean metric on $S_r$ and let $\lambda_1$ and $\lambda_2$ be
the eigenvalues of $A$. Then
\begin{equation}\label{2nd-small-e2}
 \begin{split}
 \lambda_1\lambda_2&=\frac{A(e_1,e_1)A(e_2,e_2)-A^2(e_1,e_2)}
  {g(e_1,e_2)g(e_2,e_2)-g^2(e_1,e_2)}\\
  &=\lf(\frac1{r^2}-\frac23
  R_{kimj}\frac{x^kx^m}{r^2}\lf(e_1(x^i)e_1(x^j)+e_2(x^i)e_2(x^j)\ri)+O(r)\ri)\\
  &\quad \times\lf(1-\frac13R_{ikmj}x^kx^m
  \lf(e_1(x^i)e_1(x^j)+e_2(x^i)e_2(x^j)\ri)+O(r^3)\ri)\\
  &=\frac1{r^2}-\frac13R_{kimj}\frac{x^kx^m}{r^2}\lf(e_1(x^i)e_1(x^j)+e_2(x^i)e_2(x^j)\ri)+O(r)\\
  &=\frac1{r^2}-\frac13 R_{km}\frac{x^kx^m}{r^2}+O(r),
\end{split}
\end{equation}
where we have used the fact that $\sum_{i}\lf(e_a(x^i)\ri)^2=1$
and $e_a(\sum_{i}(x^i)^2)=0$ on $S_r$ for $a=1, 2$, and the fact
that
$$
e_1(x^i)e_1(x^j)+e_2(x^i)e_2(x^j)=\delta_{ij}-\frac{x^ix^j}{r^2}.
$$
Hence by the Gauss equation, for $x\in S_r$,
\begin{equation}\label{2nd-small-e3}
   \begin{split}
 K(x)&=\lambda_1\lambda_2+\frac12h^{ik}h^{jl}R_{ijkl}(x)\\
 &=\frac1{r^2}-\frac13
 R_{km}\frac{x^kx^m}{r^2}+\frac12h^{ik}h^{jl}R_{ijkl}(p)+O(r)\\
 &=\frac1{r^2}+\frac12 R(p)-\frac43
 R_{ij}(p)\frac{x^ix^j}{r^2}+O(r)
\end{split}
\end{equation}
where $h^{ij}=g^{ij}-n^in^j$. On the other hand, for $x\in S_r$
\begin{equation}\label{2nd-small-e3}
   \begin{split}
 H(x)&=h^{ij}A_{ij}\\
 &=h^{ij}\lf(\frac{\delta_{ij}}r-
 \frac23R_{kimj}\frac{x^kx^m}r\ri)+O(r^2)\\
&=\lf(\delta_{ij}-\frac13
R_{ikmj}x^kx^m-\frac{x^ix^j}{r^2}\ri)\lf(\frac{\delta_{ij}}r -
 \frac23R_{kimj}\frac{x^kx^m}r\ri)+O(r^2)\\
 &=\frac2r-\frac13 R_{ij}\frac{x^ix^j}r+O(r^2)
\end{split}
\end{equation}
where we have used the fact that $h^{ij}x^ix^j=0$.

It remains to prove the last assertion. Let $Z_r$ be the embedding
as in the proof of  Lemma \ref{normalexpansion}. One may conclude
that by an isometry of $\R^3$, we have $||Z_r-Id||_2=O(r^2)$, where
$Id$ is the identity map of $\mathbb{S}^2$. Let $H_r$ and $K_r$ be
the mean curvature and Gauss curvature of $Z_r( \mathbb{S}^2)$. Let
$e_1$ and $e_2$ be orthonormal frames on $\mathbb{S}^2$ with respect
to the standard metric, then the metric tensor $h$ and the second
fundamental form $A$ of the surface $Z_r( \mathbb{S}^2)$ satisfies:
\begin{equation}
   h(e_a,e_b)=\delta_{ab}+\alpha_{ab},\ \  A(e_a,e_b)=\delta_{ab}+\beta_{ab},
\end{equation}
where $\alpha_{ab}=O(r^2)$ and $\beta_{ab}=O(r^2)$.  Hence we have
$$
K_r=1-\alpha_{11}-\alpha_{22}+\beta_{11}+\beta_{22}+O(r^4),\ \
H_r=2-\alpha_{11}-\alpha_{22}+\beta_{11}+\beta_{22}+O(r^4).
$$
After rescale to an embedding of   $(S_r,g|_{S_r})$ in $\R^3$, we
conclude that
$$
K=\frac1
{r^2}\lf(1-\alpha_{11}-\alpha_{22}+\beta_{11}+\beta_{22}\ri)+O(r^2)
$$
and
$$
H_0=\frac1
{r}\lf(2-\alpha_{11}-\alpha_{22}+\beta_{11}+\beta_{22}\ri)+O(r^3).
$$
From these and \eqref{Gauss-mean-e1}, \eqref{Gauss-mean-e3}
follows.
\end{proof}

We are ready to prove the following (Theorem
\ref{smallspherelimit-in}):
\begin{theo}\label{smallspherelimit}
Let $(N,g)$ be a Riemannian manifold of dimension three, $p$ be a
fixed interior point on $N$, and $S_r$ be the geodesic sphere of
radius $r$ center at $p$. For $r$ small enough, we have
\begin{equation}\label{smallspherelimit-e1}
m_{BY}(S_r)=\frac{r^3}{12}R(p)+\frac{r^5}{1440}\lf[24|Ric|^2
(p)-13R^{2}(p) + 12 \Delta R (p)\ri]+O(r^6),
\end{equation}
here, $\Delta$ is Laplacian operator of $(M,g)$.
\end{theo}
\begin{proof} For $r$ small, let $Z$ be the isometric embedding of
$(S_r,g|_{S_r})$ in $\R^3$ as in Lemma \ref{normalexpansion} and
let $H_0$ be the mean curvature of $Z(S_r)$ in $\R^3$. Let
$$k_0=
\frac R2-\frac 43 R_{ij}\frac{x^ix^j}{r^2},\ h_1=rk_0,\
n_3=\frac{r^3}6\lf(\frac R2-2R_{ij}\frac{x^ix^j}{r^2}\ri).$$
 By Lemmas   \ref{Gauss-mean} and \ref{normalexpansion}, we have
$$
K=\frac 1{r^2}+k_0+O(r),\  H_0=\frac 2 r+h_1+O(r^2), \ Z\cdot
n=r+n_3+O(r^4).
$$
 As in the proof of Theorem \ref{thmlc} in section
2, by one of the Minkowski integral formulae \cite[Lemma
6.2.9]{KLBG} and Lemma \ref{Gauss-mean}, we have

\begin{equation}\label{5-meancurvint-1}
    \begin{split}
\int_{S_r}H_0d\Sigma_r&=2 \int_{S_r}K Z\cdot n
d\Sigma_r\\
&=\frac1{r^2}\int_{S_r}Z\cdot
nd\Sigma_r+2\int_{S_r}(K-\frac 1{r^2})(r+n_3) d\Sigma_r+O(r^6)\\
&=6r^{-2}V_0(r)+2r\int_{S_r}(K-\frac 1{r^2}) d\Sigma_r+2 \int_{S_r}k_0n_3d\Sigma_r+O(r^6)\\
 &=6r^{-2}V_0(r)+8\pi r-\frac{2\mathcal{A}(r)}{r}+2 \int_{S_r}k_0n_3d\Sigma_r+O(r^6),\\
\end{split}
\end{equation}
where $V_0(r)$ is the volume inside  $Z(S_r)$ in $\R^3$.

By another Minkowski integral formula, we obtain
\begin{equation}
    \begin{split}
2\mathcal{A}(r)&=\int_{S_r}H_0Z\cdot
nd\Sigma_r\\
&=\int_{S_r}\frac 2rZ\cdot nd\Sigma_r+\int_{S_r}(H_0-\frac
2r)(r+n_3)d\Sigma_r+O(r^7)\\
&=6r^{-1}V_0(r)+r\int_{S_r}H_0d\Sigma_r-2\mathcal{A}(r)+\int_{S_r}h_1
n_3d\Sigma_r+O(r^7)\\
&=6r^{-2}V_0(r)+r\int_{S_r}H_0d\Sigma_r-2\mathcal{A}(r)+r\int_{S_r}k_0
n_3d\Sigma_r+O(r^7).
\end{split}
\end{equation}
Hence
\begin{equation}\label{5-meancurvint-2}
    \int_{S_r}H_0d\Sigma_r=-6r^{-2}V_0(r)+4r^{-1}\mathcal{A}(r)- \int_{S_r}k_0
n_3d\Sigma_r+O(r^6).
\end{equation}
By \eqref{5-meancurvint-1} and \eqref{5-meancurvint-2}, we have
\begin{equation}\label{5-meancurvint-3}
\begin{split}
   \int_{S_r}H_0d\Sigma_r&=4\pi r+r^{-1}\mathcal{A}(r)
   +\int_{S_r}k_0n_3d\Sigma_r- \frac1{2}\int_{S_r}k_0
n_3d\Sigma_r+O(r^6)\\
&=8\pi r+\frac{A_4+A_6}r+\frac12\int_{S_r}k_0 n_3d\Sigma_r
+O(r^6)\\
\end{split}
\end{equation}
where we have used Lemma \ref{expansionofarea}. Combining this
with Lemma \ref{expansionofarea-cor} we have
\begin{equation}\label{BY-small-e1}
   \begin{split}
 \int_{S_r}(H_0-H)d\Sigma_r&=-\frac{3A_4+5A_6}r
 +\frac12\int_{S_r}k_0 n_3d\Sigma_r
+O(r^6).
\end{split}
\end{equation}
Now by \eqref{4area-s3-1} and \eqref{4Areafricci11}
\begin{equation}\label{BY-small-e2}
    \begin{split}
\int_{S_r}k_0 n_3d\Sigma_r&=\frac{r^3}6\int_{S_r}\lf(\frac
R2-\frac 43 R_{ij}\frac{x^ix^j}{r^2}\ri) \lf(\frac
R2-2R_{ij}\frac{x^ix^j}{r^2}\ri)d\Sigma_r\\
&=\frac{r^3}6\int_{S_r}\lf(\frac R2-\frac 43
R_{ij}\frac{x^ix^j}{r^2}\ri) \lf(\frac
R2-2R_{ij}\frac{x^ix^j}{r^2}\ri)d\Sigma_r^0+O(r^6)\\
&= \frac{\pi r^5}{270}\lf(64|Ric|^2-23R^2\ri).
\end{split}
\end{equation}
The theorem follows from \eqref{BY-small-e1},\eqref{BY-small-e2}
and Lemma \ref{expansionofarea}.
\end{proof}

As a corollary, we have

\begin{coro} \label{cobyd5pmt1}
With the assumptions and notation as in Theorem
\ref{smallspherelimit},  suppose $R\ge 0$ in a neighborhood of
$p$,  then
 \begin{equation}\label{nonnegative}
 \lim_{r\rightarrow 0}\frac{m_{BY}(S_r)}{r^5}\ge 0.
 \end{equation}
Equality holds if and only if  $(N,g)$ is flat at $p$ and $R$
vanishes up to second order at $p$.
\end{coro}
\begin{proof} By the result of \cite{ST02} on the positivity of
Brown-York mass, we know that \eqref{nonnegative} is true.
However, in this special case, one can deduce this from the
theorem. In fact, if $R(p)>0$, then by
\eqref{smallspherelimit-e1}, we have
$$
\lim_{r\rightarrow 0}\frac{m_{BY}(S_r)}{r^5}=\infty>0.
$$
In case $R(p)=0$, then $R(p)$ is a minimum of $R$ because $R\ge0$.
It is easy to see that \eqref{nonnegative} is still true.

It is obvious that if $(N,g)$ is flat at $p$ and $R$ vanishes up
to second order at $p$, then equality holds in
\eqref{nonnegative}. Conversely, if the equality holds in
\eqref{nonnegative}, then we must have $R(p)=0$, $\nabla R(p)=0$,
$\Delta R(p)=0$ and $|Ric|(p)=0$. Since $R$ has a minimum at $p$,
the Hessian of $R$ has nonnegative eigenvalues. So the Hessian of
$R$ must be zero at $p$ because $\Delta R(p)=0$. Moreover, since
$N$ has dimension three, $|Ric|(p)=0$ implies that $(N,g)$ is flat
at $p$.
\end{proof}
\begin{rema} From the proof, it is easy to see that
\eqref{nonnegative} is true if $R(p)=0$ and $\Delta R(p)\ge0$ and
 the equality holds only if $g$ is flat at $p$.
\end{rema}

One should compare the corollary to the following fact: If $M$ is
an asymptotically flat manifold with nonnegative scalar curvature,
suppose the Brown-York mass of the coordinate  spheres converge to
zero, then $M$ must be the Euclidean space. This follows from
Theorem \ref{thmlc}  and the Positive Mass Theorem in Scheon-Yau
\cite{SYL79}, Witten \cite{Wit81}.

For the expansion of the  Hawking mass, we have:

\begin{theo}\label{hawkingmass}
With  the same notations and assumptions in Theorem
\ref{smallspherelimit}, we have
\begin{equation}
\begin{split}
m_H (S_r)=\frac{r^3}{12}R(p)+\frac{r^5}{720}\lf(6\Delta R (p)-5
R^2(p)\ri)+O(r^6).
\end{split}
\end{equation}
\end{theo}
\begin{proof} By Lemma \ref{Gauss-mean}, we have
$$
H=\frac2r+H_1+O(r^2), $$ where $H_1=-\frac 1{3r}R_{ij}{x^ix^j}. $
Hence
\begin{equation}\label{4expHsq-1}
   \begin{split}
    H^2 & =-4 r^{-2}+4r^{-1}H+H^2_1+O\lf(r^3\ri).
\end{split}\nonumber
\end{equation}
Then,
\begin{equation}\label{4expansion-2}
   \begin{split}
    \int_{S_r}H^2d\Sigma_r & =-\frac{4\mathcal{A}(r)}{r^2}+\frac4r\int_{S_r}Hd\Sigma_r\\
    &+\int_{S_r}H_1^2d\Sigma_r^0+O(r^5)
     \\
& =-\frac{4(4\pi r^2+A_4+A_6)}{r^2}
       +\frac4 r\cdot\lf(8\pi r+\frac{4A_4}r+\frac{6A_6}r\ri)
       +\int_{S_r}H_1^2d\Sigma_r^0+O(r^5)\\
       &=16\pi+\frac{12A_4}{r^2}+\frac{20A_6}{r^2}+\int_{S_r}H_1^2d\Sigma_r^0+O(r^5).
\end{split}\nonumber
\end{equation}
Hence
\begin{equation}\label{4expansion-3}
16\pi-\int_{S_r}H^2d\Sigma_r=-\frac{12A_4}{r^2}-\frac{20A_6}{r^2}
-\int_{S_r}H_1^2d\Sigma_r^0+O(r^5).\nonumber
\end{equation}
On the other hand,
\begin{equation}\label{4expansion-4}
\begin{split}
\frac{\mathcal{A}^\frac12(r)}{(16\pi)^\frac32}&=\frac{2\pi^\frac12
r}{(16\pi)^\frac32}\lf(1+\frac{A_4}{8\pi r^2}+O(r^4)\ri)\\
&=\frac{r}{32\pi}\lf(1+\frac{A_4}{8\pi r^2}+O(r^4)\ri).
\end{split}\nonumber
\end{equation}
So
\begin{equation}\label{4Hawking-s1}
m_H(S_r)=-\frac{3A_4}{8\pi r}-\frac{5A_6}{8\pi
r}-\frac{r}{32\pi}\int_{S_r}H_1^2d\Sigma_r^0-\frac{3A_4^2}{64\pi^2
r^3}+O(r^6).\nonumber
\end{equation}
By \eqref{4Areafricci11} and Lemma \ref{expansionofarea}, the
result follows.
\end{proof}
Hence the expansion of the Brown-York mass and the Hawking mass
are equal up to order $r^3$. However, they differ on the term of
order $r^5$.

As in the case of large-sphere limit, we can compare $V(r)$ and
$V_0(r)$, where $V(r)$ is the volume of the geodesic ball of
radius $r$ at $p$ and $V_0(r)$ is the volume of the region bounded
by $S_r$ when embedded in $\R^3$.

\begin{theo}\label{volcompare}
With the above notations, we have:
\begin{equation} \label{vbysv0mv2}
\begin{split}
V_0(r)-V(r)&=-\frac{\pi}{15}Rr^5+\frac{\pi
r^7}{5670}(173R^2-454|Ric|^2-27\Delta R(p))+O\lf(r^8\ri).
\end{split}
\end{equation}
\end{theo}
\begin{proof} By \eqref{5-meancurvint-1} and \eqref{5-meancurvint-2}
\begin{equation}
\int_{S_r}H_0 d\Sigma_r=6r^{-2}V_0(r)+8\pi
r-\frac{2\mathcal{A}(r)}{r}+2
\int_{S_r}k_0n_3d\Sigma_r+O(r^6),\nonumber
\end{equation}
and

\begin{equation}\label{4integralho-3}
\begin{split}
\int_{S_r}H_0d\Sigma_r&=-6r^{-2}V_0(r)+4r^{-1}\mathcal{A}(r)-\int_{S_r}k_0n_3d\Sigma_r,
\end{split}\nonumber
\end{equation}
where $k_0$ is as in \eqref{5-meancurvint-1}.

We have
\begin{equation} \label{vbysv01}
\begin{split}
V_0(r)&=
\frac{r}{2}\mathcal{A}(r)-\frac{r^2}{4}\int_{S_r}k_0n_3d\Sigma^0_r-\frac{2}{3}\pi
r^3+O\lf(r^8\ri)\\
&=\frac{4}{3}\pi
r^3+\frac{r}{2}A_4+\frac{r}{2}A_6-\frac{r^2}{4}\int_{S_r}k_0n_3d\Sigma^0_r+O\lf(r^8\ri).
\end{split}\nonumber
\end{equation}
On the other hand,
\begin{equation} \label{vbysv1}
\begin{split}
V(r)&=\int_0^r\mathcal{A}(t)dt\\
&=\frac{4}{3}\pi r^3+\int_0^rA_4dt+\int_0^rA_6dt+O\lf(r^8\ri)\\
&=\frac{4}{3}\pi r^3+\frac r5 A_4+\frac r7 A_6+O\lf(r^8\ri).
\end{split}\nonumber
\end{equation}
Hence
\begin{equation} \label{vbysv0mv1}
\begin{split}
V_0(r)-V(r)&=\frac{3}{10}r A_4+\frac{5}{14}r
A_6-\frac{r^2}{4}\int_{S_r}k_0n_3d\Sigma^0_r+O\lf(r^8\ri).
\end{split}\nonumber
\end{equation}
By \eqref{BY-small-e2} and Lemma \ref{expansionofarea}, the result
follows.
\end{proof}
By Theorem \ref{volcompare}, we see that if scalar curvature is
positive at $p$, then $V_0 (r)< V(r)$, for sufficiently small $r$.
 More precisely,
 \begin{coro}\label{vcpdifss}
With the assumptions and notations as in Theorem \ref{volcompare},
suppose $R\geq 0$ in a neighborhood of $p$, then
\[
\lim_{r\to 0}\frac{V_0(r)-V(r)}{r^7}\le 0.\] Equality holds if and
only if  $(N,g)$ is flat at $p$ and $R$ vanishes up to second order
at $p$.
 \end{coro}
\begin{proof}
 Similar to the argument of Corollary
\ref{cobyd5pmt1}, one can derive the result from Theorem
\ref{volcompare}.
\end{proof}

\end{document}